\documentclass[a4paper, 11pt]{article}
\usepackage{amsmath, amssymb,amscd, amsthm}
\usepackage[mathscr]{eucal}
\usepackage{graphics}
\usepackage{fullpage}
\usepackage{url}
\newcommand\cyr{%
 \renewcommand\rmdefault{wncyr}%
 \renewcommand\sfdefault{wncyss}%
 \renewcommand\encodingdefault{OT2}%
\normalfont\selectfont} \DeclareTextFontCommand{\textcyr}{\cyr}

\newtheorem{theorem}{Theorem}
\newtheorem{lemma}[theorem]{Lemma}

\newtheorem{claim}[theorem]{Claim}

\long\def\symbolfootnote[#1]#2{\begingroup%
\def\thefootnote{\fnsymbol{footnote}}\footnote[#1]{#2}\endgroup}

\newtheorem*{thm1}{Theorem 1}
\newtheorem*{thm2}{Theorem 2}

\title{Endpoint restriction estimates for the paraboloid over finite fields}

\author{Allison Lewko \thanks{Supported by a National Defense Science and Engineering Graduate Fellowship} \and Mark Lewko}

\date{}

\begin{document}

\maketitle

\begin{abstract} We prove certain endpoint restriction estimates for the paraboloid over finite fields in three and higher dimensions. Working in the bilinear setting, we are able to pass from estimates for characteristic functions to estimates for general functions while avoiding the extra logarithmic power of the field size which is introduced by the dyadic pigeonhole approach. This allows us to remove logarithmic factors from the estimates obtained by Mockenhaupt and Tao in three dimensions and those obtained by Iosevich and Koh in higher dimensions.
\end{abstract}

\symbolfootnote[0]{2000 Mathematics Subject Classification. 42B10}

\section{Introduction}
Let $S$ denote a hypersurface in $\mathbb{R}^{n}$ with measure $d\sigma$. The restriction problem for $S$ is to determine for which pairs of $(p,q)$ does there exist an inequality of the form
\begin{equation}\label{restrictionIneqaulity}
||\hat{f}||_{L^{p'}(S,d\sigma)} \leq C ||f||_{L^{q'}(\mathbb{R}^n)}.
\end{equation}
We note that the left-hand side is not necessarily well-defined since we have restricted the function $\hat{f}$ to the hypersurface $S$, a set of measure zero in  $\mathbb{R}^{n}$. However, if we can establish this inequality for all Schwartz functions $f$, then the operator that restricts $\hat{f}$ to $S$ can be defined whenever $f \in L^{q}$. In the Euclidean setting, the restriction problem has been extensively studied for many surfaces. In particular, it has been observed that restriction estimates are intimately connected to questions about certain partial differential equations as well as problems in geometric measure theory such as the Kakeya conjecture. The restriction conjecture states sufficient conditions on $(p,q)$ for the above inequality to hold. In the case of the sphere and paraboloid, the question is open in dimensions three and higher. For a survey of restriction results in the Euclidean setting, see \cite{T}.

In \cite{MT}, Mockenhaupt and Tao initiated the study of the restriction phenomena in the finite field setting. This is motivated by both the similarities and the differences between the Euclidean and finite field settings, which suggest that studying restriction phenomena in finite fields may yield insights which are portable to the Euclidean setting, but also that the problems in the finite field setting present unique and independently interesting challenges. In addition, these problems in the finite field setting are closely related to other areas of mathematics, and particularly seem amenable to the use of combinatorial techniques.

We now introduce some notation to formally define the problem. We let $F$ denote a finite field of characteristic $p >2$. We let $S^{1}$ denote the unit circle in $\mathbb{C}$ and define $e: F \rightarrow S^1$ to be a non-principal character of $F$. For example, when $F = \mathbb{Z}/p \mathbb{Z}$, we can set $e(x) := e^{2\pi i x/p}$. We will be considering the vector space $F^n$ and its dual space $F_*^n$. Following the conventions of \cite{MT}, we think of $F^n$ as endowed with the counting measure $dx$ which assigns mass 1 to each point and $F_*^n$ as endowed with the normalized counting measure $d\xi$ which assigns mass $|F|^{-n}$ to each point (where $|F|$ denotes the size of $F$, so the total mass is equal to 1 here). To be clear in our calculations, we will always include the appropriate powers of $|F|$ explicitly.

For a complex-valued function $f$ on $F^n$, we define its Fourier transform $\hat{f}$ on $F_*^n$ by:
\[\hat{f}(\xi) := \sum_{x \in F^n} f(x) e(-x \cdot \xi).\]
For a complex-valued function $g$ on $F_*^n$, we define its inverse Fourier transform $g^{\vee}$ on $F^n$ by:
\[g^{\vee}(x) := \frac{1}{|F|^n} \sum_{\xi \in F_*^n} g(\xi) e(x\cdot \xi).\]
It is easy to verify that $(\hat{f})^\vee = f$ and $\widehat{(g^{\vee})} = g$.

We define the paraboloid $\mathcal{P} \subset F_*^n$ as: $\mathcal{P} := \{(\gamma, \gamma \cdot \gamma): \gamma \in F_*^{n-1}\}$. This is endowed with the normalized ``surface measure" $d\sigma$ which assigns mass $|\mathcal{P}|^{-1}$ to each point in $\mathcal{P}$. We note that $|\mathcal{P}| = |F|^{n-1}$.
For a function $f: \mathcal{P} \rightarrow \mathbb{C}$, we define the function $(f d\sigma)^\vee: F^n \rightarrow \mathbb{C}$ as follows:
\[(f d\sigma)^\vee (x) := \frac{1}{|\mathcal{P}|} \sum_{\xi \in \mathcal{P}} f(\xi) e(x \cdot \xi).\]

For a complex-valued function $f$ on $F^n$ and $q \in [1, \infty)$, we define \\$||f||_{L^q(F^n, dx)} := \left( \sum_{x \in F^n} |f(x)|^q \right)^{\frac{1}{q}}$. For a complex-valued function $f$ on $\mathcal{P}$, we similarly define $||f||_{L^q(\mathcal{P},d\sigma)} := \left( \frac{1}{|\mathcal{P}|} \sum_{\xi \in \mathcal{P}} |f(\xi)|^q \right)^{\frac{1}{q}}$. (These are the standard definitions of the $L^q$ norms, and hence they satisfy the usual properties of norms.)

Now we define a restriction inequality to be an inequality of the form

\begin{equation}\label{restrictionIneqaulityFF}
||\hat{f}||_{L^{p'}(S,d\sigma)} \leq \mathcal{R}(p\rightarrow q) ||f||_{L^{q'}(F^n)},
\end{equation}
where $\mathcal{R}(p\rightarrow q)$ denotes the best constant such that the above inequality holds. Here $p'$ and $q'$ denote the conjugate exponents of $p$ and $q$ respectively (i.e. $\frac{1}{p} + \frac{1}{p'} = 1$). By duality, this is equivalent to the following extension estimate:

\begin{equation}\label{extentionIneqaulityFF}
||(f d\sigma)^\vee||_{L^q(F^n, dx)} \leq \mathcal{R}(p\rightarrow q) ||f||_{L^p(S,d\sigma)}.
\end{equation}

We will only be considering the case of $S = \mathcal{P}$. We will use the notation $X \ll Y$ to denote that quantity $X$ is at most a constant times quantity $Y$, where this constant may depend on the dimension $n$ but not on the field size, $|F|$. For a finite field $F$, the constant $\mathcal{R}(p\rightarrow q)$ will always be finite. The restriction problem in this setting is to determine for which $(p,q)$ can we upper bound $\mathcal{R}(p\rightarrow q)$ independently of $|F|$ (i.e. for which $(p,q)$ does $\mathcal{R}(p \rightarrow q) \ll 1$ hold).

Mockenhaupt and Tao \cite{MT} solved this problem for the paraboloid in two dimensions.  In three dimensions, we require $-1$ not be a square in $F$. For such $F$, they showed that $\mathcal{R}(8/5+\epsilon \rightarrow 4) \ll 1$ and $\mathcal{R}(2 \rightarrow \frac{18}{5}+\epsilon) \ll 1$ for every $\epsilon>0$. When $\epsilon=0$, their bounds were polylogarithmic in $|F|$. Mockenhaupt and Tao's argument for the $\mathcal{R}(8/5 \rightarrow 4)$ estimate proceeded by first establishing the estimate for characteristic functions. Here one can expand the $L^4$ norm and reduce the problem to combinatorial estimates. A well-known dyadic pigeonhole argument then allows one to pass back to general functions at the expense of a logarithmic power of $|F|$. The work of Iosevich and Koh in \cite{IK}, \cite{IKs}, and \cite{IKq} follows the same approach: first proving restriction estimates for characteristic functions in the finite field setting, and then incurring an extra logarithmic power of $|F|$ in the general estimates obtained through the dyadic pigeonhole argument.

We introduce a method for obtaining general estimates which avoids the polylogarithmic cost of the dyadic pigeonhole technique. Our argument begins by rewriting the $L^4$ norm as $||(fd\sigma)^{\vee}||_{L^4}=||(fd\sigma)^{\vee}(fd\sigma)^{\vee}||_{L^2}^{1/2}$. We then adapt the arguments of \cite{MT} and \cite{IK} to the bilinear variant $||(fd\sigma)^{\vee}(gd\sigma)^{\vee}||_{L^2}^{1/2}$ in the case that $f$ and $g$ are characteristic functions. The key point is that we allow $f$ and $g$ to be different characteristic functions - this is what makes our method more powerful than the standard dyadic pigeonhole technique.

To obtain estimates for arbitrary functions $f$, we can assume that $f$ is non-negative real-valued and decompose $f$ as a linear combination of characteristic functions, where the coefficients are negative powers of two (we can do this without loss of generality by adjusting only the constant of our bound). We can then employ the triangle inequality to upper bound $||(fd\sigma)^{\vee}||_{L^4}$ by a double sum of terms like $||(\chi_j d\sigma)^{\vee}(\chi_k d\sigma)^{\vee}||_{L^2}^{1/2}$, where $\chi_j$ and $\chi_k$ are characteristic functions, weighted by negative powers of two.
We then apply our bilinear estimate for characteristic functions to these inner terms and use standard bounds on sums to obtain the final estimates.

Our method yields the following theorems:

\begin{theorem}\label{thm:dim3} For the paraboloid in $3$ dimensions with $-1$ not a square, we have $\mathcal{R}(8/5 \rightarrow 4) \ll 1$ and $\mathcal{R}(2 \rightarrow \frac{18}{5}) \ll 1$.
\end{theorem}

This improves upon Proposition 5.2 in \cite{MT} by removing the logarithmic power of $|F|$. While our argument is elementary and does not use multilinear interpolation, the proof is certainly in the spirit of real interpolation. The second estimate, $\mathcal{R}(2 \rightarrow \frac{18}{5}) \ll 1$, follows from the first using the machinery of \cite{MT} without modification. After discovering our proof, we learned that in unpublished work Bennett, Carbery, Garrigos, and Wright \cite{BCGW} have independently obtained the end-point results in the $3$-dimensional case. Their argument proceeds rather differently than ours and it is unclear if their argument can be extended to the higher dimensional settings. In higher dimensions, we prove:

\begin{theorem}\label{thm:RestrictionForHigher} For the paraboloid in $n$ dimensions when $n \geq 4$ is even or when $n$ is odd and $|F| = q^m$ for a prime $q$ congruent to 3 modulo 4 such that $m(n-1)$ is not a multiple of 4, we have $\mathcal{R}(\frac{4n}{3n-2} \rightarrow 4) \ll 1$ and $\mathcal{R}(2 \rightarrow \frac{2n^2}{n^2-2n+2}) \ll 1$.
\end{theorem}

This improves upon Theorems 1, 2, and 3 of \cite{IK} by removing the logarithmic power of $|F|$. We will only prove $\mathcal{R}(\frac{4n}{3n-2} \rightarrow 4) \ll 1$ here. The estimate $\mathcal{R}(2 \rightarrow \frac{2n^2}{n^2-2n+2}) \ll 1$ follows from the previous estimate from the arguments of \cite{IK}.

We have restricted our attention to the case of the parabaloid, however our methods are more generic and likely can be combined with the arguments of \cite{IKs} and \cite{IKq} for the cases of spheres and more general quadratic surfaces, respectively.

\section{A Restriction Theorem for the Paraboloid in $F_*^3$}
We first prove Theorem \ref{thm:dim3} for the paraboloid in $F_*^3$, restated below in an equivalent formulation:

\begin{thm1} For every function $f: \mathcal{P} \rightarrow \mathbb{C}$, we have that:
\[||(f d\sigma)^\vee ||_{L^4(F^3, dx)} \leq C ||f||_{L^{8/5}(\mathcal{P}, d\sigma)}\]
for some constant $C$.
\end{thm1}
More concretely, we show below that $C =4 \left( \frac{1}{\left(1-2^{-1/5}\right) \left(1-2^{-2/5}\right)}  + \frac{1}{1-2^{-3/5}} \right)^{1/2}\leq 6$ suffices.
\vspace*{0.5cm}

We start by following the strategy of \cite{MT}, generalizing to the bilinear setting. We employ the following two lemmas. The first one is standard.
\begin{lemma} \label{lem:incidence} Let $P$ be a collection of points in $F^2$, and let $L$ be a collection of lines in $F^2$. Then:
\[|\{(p, \ell) \in P \times L: p \in \ell\}| \leq \min \left( |P|^{1/2}|L| + |P|, |P||L|^{1/2} + |L|\right).\]
\end{lemma}

\begin{lemma} \label{lem:bilinear} We let $A, B \subseteq \mathcal{P}$ be arbitrary subsets of $\mathcal{P}$. We define $\chi_A, \chi_B$ to be the corresponding characteristic functions from $\mathcal{P}$ to $\{0,1\}$. Then:
\[\left|\left|(\chi_A d\sigma)^\vee (\chi_B d\sigma)^\vee \right|\right|_{L^2(F^3,dx)}^2 \leq 2 \cdot \frac{|F|^{3}}{|\mathcal{P}|^4}\cdot \min \left(|A|^{1/2}|B|^2 + |A||B|, |A||B|^{3/2}+ |B|^2\right) .\]
\end{lemma}

\begin{proof} By definition of the $L^2$ norm, we have:
\[\left|\left|(\chi_A d\sigma)^\vee (\chi_B d\sigma)^\vee \right|\right|_{L^2(F^3,dx)}^2 =  \sum_{x \in F^3} \left|(\chi_A d\sigma)^\vee (x) (\chi_B d\sigma)^\vee (x)\right|^2.\]
Using the definitions of $(\chi_A d\sigma)^\vee$ and $(\chi_B d\sigma)^\vee$, we can expand this as:
\[ = \sum_{x \in F^3} \left| \frac{1}{|\mathcal{P}|} \sum_{\xi_1 \in \mathcal{P}} \chi_A(\xi_1) e(x\cdot \xi_1) \cdot \frac{1}{|\mathcal{P}|} \sum_{\xi_2 \in \mathcal{P}} \chi_B(\xi_2) e(x\cdot \xi_2)\right|^2.\]
We can rewrite this as:
\[\frac{1}{|\mathcal{P}|^4} \sum_{x\in F^3} \left| \sum_{\xi_1 \in \mathcal{P}} \chi_A(\xi_1) e(x \cdot \xi_1) \cdot \sum_{\xi_2 \in \mathcal{P}} \chi_B(\xi_2) e(x\cdot \xi_2)\right|^2.\]

For any complex number $z$, we note that $|z|^2 = z \overline{z}$, where $\overline{z}$ denotes the complex conjugate of $z$. This allows us to express the above quantity as:
\[ = \frac{1}{|\mathcal{P}|^4} \sum_{x \in F^3} \sum_{a,b,c,d \in \mathcal{P}} \chi_A(a) \chi_B(b) \chi_A(c) \chi_B(d) e(x\cdot a) e(x \cdot b) e(-x\cdot c) e(-x \cdot d).\]
Here, we have used that $\overline{\chi_A} = \chi_A$, $\overline{\chi_B} = \chi_B$, and $\overline{e(x\cdot \xi)} = e(-x\cdot \xi)$.

Since these are finite sums, we can interchange their order to obtain:
\[ = \frac{1}{|\mathcal{P}|^4} \sum_{\substack{a,c \in A \\ b,d \in B}} \sum_{x \in F^3}e(x \cdot (a+b-c-d)).\]
This inner sum will be equal to zero except when $a+b = c+d$. When this occurs, the inner sum will equal $|F|^3$. Thus, we get:
\[ = \frac{1}{|\mathcal{P}|^4} \sum_{\substack{a+b = c+d \\ a,c \in A \\ b,d \in B}} |F|^3 = \frac{|F|^3}{|\mathcal{P}|^4} \sum_{\substack{a+b = c+d \\ a,c \in A \\ b,d \in B}} 1.\]

Using the fact that $A \subseteq \mathcal{P}$, we observe:
\[\sum_{\substack{a+b = c+d \\ a,c \in A \\ b,d \in B }} 1 = \sum_{\substack{a-d = c-b \\ a,c \in A \\ b,d \in B }} 1 \leq \sum_{\substack{a-d = c-b \\ a \in A \\ b,d \in B \\ c \in \mathcal{P}}} 1.\]
This can be bounded above by:
\[\leq |B| \cdot \max_{b \in B} \sum_{\substack{a-d = c-b \\ a \in A \\ d \in B \\ c \in \mathcal{P}}} 1 \leq |B| \cdot \max_{ b\in \mathcal{P}} \sum_{\substack{a-d = c-b \\ a \in A \\ d\in B \\ c \in \mathcal{P}}} 1 .\]

We now consider the quantity inside the maximum for an arbitrary, fixed $b \in \mathcal{P}$. To bound this, we will use the Galilean transformation $g_{\delta}: \mathcal{P} \rightarrow \mathcal{P}$, which is defined for each $\delta \in F_*^{2}$ by:
\[g_{\delta}(\gamma, \tau) := (\gamma+ \delta, \tau + 2 \gamma \cdot \delta + \delta \cdot \delta),\]
where $(\gamma, \tau) \in F_*^{2} \times F_* = F_*^3$. We note that for each $\delta \in F_*^{2}$, this is a bijective map from $\mathcal{P}$ to itself.

\begin{claim} We write $b \in \mathcal{P}$ as $b = (\nu, \nu \cdot \nu)$, for $\nu \in F_*^{2}$. We also define $A':= g_{-\nu} (A)$ and $B' := g_{-\nu}(B)$. We then have:
\[\sum_{\substack{a-d = c-b\\ a \in A \\ d \in B \\ c\in \mathcal{P}}} 1 = \sum_{\substack{a'-d' \in \mathcal{P} \\ a' \in A' \\ d' \in B'}} 1.\]
\end{claim}

\begin{proof}[Proof of Claim] We first observe that
\[\sum_{\substack{a-d = c-b\\ a \in A \\ d \in B \\ c\in \mathcal{P}}} 1  = \sum_{\substack{a-d+b \in \mathcal{P} \\ a \in A \\ d \in B}}1.\]
We will show that for $a \in A, d \in B$, $a-d + b \in \mathcal{P}$ if and only if $g_{-\nu}(a) - g_{-\nu} (d) \in \mathcal{P}$.

We can write $a$ as $(\alpha, \alpha \cdot \alpha)$ for some $\alpha \in F_*^{2}$, and $d$ as $(\eta, \eta \cdot \eta)$ for some $\eta \in F_*^{2}$. We can then compute $g_{-\nu}(a) - g_{-\nu}(d)$ as:
\[g_{-\nu}(a) - g_{-\nu}(d) = (\alpha - \eta, \alpha \cdot \alpha - \eta \cdot \eta - 2(\alpha - \eta)\cdot \nu).\]
This will be an element of $\mathcal{P}$ if and only if:
\[(\alpha - \eta) \cdot (\alpha - \eta) = \alpha \cdot \alpha - \eta \cdot \eta - 2(\alpha - \eta)\cdot \nu,\]
which is equivalent to:
\[ \eta \cdot \eta -\alpha \cdot \eta + (\alpha - \eta) \cdot \nu = 0.\]

Now, $a-d + b \in \mathcal{P}$ holds if and only if:
\[(\alpha - \eta + \nu) \cdot (\alpha - \eta + \nu) = \alpha \cdot \alpha - \eta \cdot \eta + \nu \cdot \nu,\]
which is also equivalent to:
\[\eta \cdot \eta - \alpha \cdot \eta + \alpha \cdot \nu - \eta \cdot \nu = 0.\]
Thus, $a-d + b \in \mathcal{P}$ if and only if $g_{-\nu}(a) - g_{-\nu} (d) \in \mathcal{P}$. We may then conclude that
\[\sum_{\substack{a-d = c-b\\ a \in A \\ d \in B \\ c\in \mathcal{P}}} 1 = \sum_{\substack{a'-d' \in \mathcal{P} \\ a' \in A' \\ d' \in B'}} 1,\]
since $g_{-\nu}$ is a bijection from $\mathcal{P}$ to $\mathcal{P}$.
\end{proof}

We are now interested in bounding the quantity \[\sum_{\substack{a'-d' \in \mathcal{P} \\ a' \in A' \\ d' \in B'}} 1.\] We note that the contribution to this sum from terms where $d = 0$ is at most $|A'| = |A|$. We note there can be no contribution from terms where $a = 0$ and $d \neq 0$, since having both of $d, -d$ in $\mathcal{P}$ is impossible for $d \neq 0$. Hence, we have:
\[\sum_{\substack{a'-d' \in \mathcal{P} \\ a' \in A' \\ d' \in B'}} 1 \leq |A| + \sum_{\substack{a'-d' \in \mathcal{P}\\ a' \in A'-\{0\} \\ d' \in B' - \{0\}}} 1.\]

We now define the sets $X_{A'} := \{\gamma \in F_*^{2} : (\gamma, \gamma \cdot \gamma) \in A' - \{0\}\}$, $X_{B'} := \{\gamma \in F_*^{2}: (\gamma, \gamma \cdot \gamma) \in B' - \{0\}\}$. Letting $a' = (x, x \cdot x)$ and $d' = (y, y \cdot y)$, we note that $a' - d' \in \mathcal{P}$ is equivalent to $x \cdot y = y\cdot y$.

For each $y \in F_*^{2}$, we can define a line in $F_*^{2}$ by $\ell(y) := \{x \in F_*^{2}: y \cdot x = y \cdot y\}$. We now prove that these lines are distinct, i.e. $y$ and $\ell(y)$ are in bijective correspondence. We suppose that $\ell(y) = \ell(y')$ for $y, y' \in F_*^2$. We note that $y \in \ell(y)$ and $y' \in \ell(y')$. Since these lines are the same, we must also have $y \in \ell(y')$ and $y' \in \ell(y)$. By definition of $\ell(y), \ell(y')$, this implies that $y \cdot y = y' \cdot y = y' \cdot y'$. Hence,
$(y-y') \cdot (y- y')  = y \cdot y - 2 y' \cdot y + y' \cdot y' = 0$.
However, since $-1$ is not a square in $F$, this implies that $y - y'$ must be the zero vector in $F_*^2$. Thus, $y = y'$.

We define $L_{B'}$ to be the collection of lines $L_{B'} := \{\ell(y): y \in X_{B'}\}$. Since these lines are distinct and $a'-d' \in \mathcal{P}$ if and only if the corresponding $x,y$ satisfy $x \in \ell(y)$, we have that:
\[\sum_{\substack{a'-d' \in \mathcal{P}\\ a' \in A'-\{0\} \\ d' \in B' - \{0\}}} 1 = \left|\{(\ell(y), x) \in L_{B'} \times X_{A'}: x \in \ell(y)\}\right|.\]

We now apply Lemma \ref{lem:incidence} to conclude:
\[\sum_{\substack{a'-d' \in \mathcal{P}\\ a' \in A'-\{0\} \\ d' \in B' - \{0\}}} 1 \leq \min \left( |X_{A'}|^{1/2} |L_{B'}| + |X_{A'}|, |X_{A'}| |L_{B'}|^{1/2} + |L_{B'}|\right).\]
Since $|L_{B'}| = |B'|$, $|X_{A'}| \leq |A'| = |A|$, and $|X_{B'}| \leq |B'| = |B|$, we have:
\[\sum_{\substack{a'-d' \in \mathcal{P}\\ a' \in A'-\{0\} \\ d' \in B' - \{0\}}} 1 \leq \min \left( |A|^{1/2} |B| + |A|, |A||B|^{1/2}+|B|\right).\]

This yields:
\[ \sum_{\substack{a+b = c+d \\ a,c \in A \\ b,d \in B}} 1 \leq |B| \left(|A| + \min \left( |A|^{1/2}|B| + |A|, |A||B|^{1/2}+|B|\right) \right).\]
Since $|B||A| \leq \min \left(|A|^{1/2}|B|^2 + |A||B|, |A||B|^{3/2}+ |B|^2\right)$, we have:
\[ \sum_{\substack{a+b = c+d \\ a,c \in A \\ b,d \in B}} 1 \leq 2 \min \left(|A|^{1/2}|B|^2 + |A||B|, |A||B|^{3/2}+ |B|^2\right).\]

Recalling that
\[\left|\left|(\chi_A d\sigma)^\vee (\chi_B d\sigma)^\vee \right|\right|_{L^2(F^3,dx)}^2 = \frac{|F|^3}{|\mathcal{P}|^4} \sum_{\substack{a+b = c+d \\ a,c \in A \\ b,d \in B}} 1,\]
we see that
\[\left|\left|(\chi_A d\sigma)^\vee (\chi_B d\sigma)^\vee \right|\right|_{L^2(F^3,dx)}^2 \leq 2 \cdot \frac{|F|^{3}}{|\mathcal{P}|^4}\cdot \min \left(|A|^{1/2}|B|^2 + |A||B|, |A||B|^{3/2}+ |B|^2\right).\]
This concludes the proof of the lemma.
\end{proof}

\begin{proof}[Proof of Theorem]
Our task reduces to proving:
\[||(f d\sigma)^\vee ||_{L^4(F^3, dx)}=|F|^{3/4-2} (  \sum_{\substack{a+b = c+d \\ a,b,c,d \in \mathcal{P}}} f(a)f(b)\overline{f(c)}\overline{f(d)}\; )^{1/4}\]
\begin{equation}\label{eq:combinequality1}
\leq C  |F|^{-5/4 }  (  \sum_{\xi \in \mathcal{P}} |f(\xi)|^{8/5} )^{\frac{5}{8}}= C ||f||_{L^{8/5}(\mathcal{P},d\sigma)}.
\end{equation}

We note that if we replace $f$ by the non-negative, real-valued function $|f|$, then the quantity
$\sum_{\substack{a+b = c+d \\ a,b,c,d \in \mathcal{P}}} f(a)f(b)\overline{f(c)}\overline{f(d)}$ cannot decrease (by the triangle inequality), and $||f||_{L^{8/5}(\mathcal{P},d\sigma)}$ remains the same. Therefore, we can assume without loss of generality that $f$ is a non-negative, real-valued function. Moreover, if we replace $|f(\xi)|$ by the smallest power of $2$ larger than $|f(\xi)|$, so that $f$ is a dyadic step-function, the left-hand side will not decrease, while the right-hand side will increase by at most a factor of $2$. Thus if we can establish inequality (\ref{eq:combinequality1}) for dyadic step functions for $C'$, the same inequality will hold for all complex-valued functions with $C$ = $2C'$). By the homogeneity of each side of (\ref{eq:combinequality1}), we may assume that $ \left(  \sum_{\xi \in \mathcal{P}} |f(\xi)|^{8/5} \right)^{\frac{5}{8}}=1$, and from the previous remarks we may assume that $f(\xi)= \sum_{j=0}^{\infty}2^{-j}\chi_{E_{j}}(\xi)$, where the $E_{j}$'s are disjoint subsets of $\mathcal{P}$. We use later the simple consequences that $\sum_{j=0}^{\infty} 2^{-j\cdot 8/5}|E_j| = 1$ and $|E_{j}| \leq 2^{j\cdot 8/5}$ for all $j$. It therefore suffices to show that:
\[ || (f d\sigma)^{\vee} ||_{L^4 (F^3,dx)}^{2} = || (f d\sigma)^{\vee} (f d\sigma)^{\vee}||_{L^2 (F^3,dx)} \leq (C')^{2}  |F|^{-5/2} .\]

We calculate
\[ || (f d\sigma)^{\vee}(f d\sigma)^{\vee}||_{L^2 (F^3,dx)} = \left|\left| \left(\sum_{j=0}^{\infty}2^{-j}\chi_{E_{j}}d\sigma\right)^{\vee}\left(\sum_{k=0}^{\infty}2^{-k}\chi_{E_{k}}d\sigma\right)^\vee \right|\right|_{L^2 (F^3,dx)} \]
\[\leq \sum_{j=0}^{\infty}2^{-j}\sum_{k=0}^{\infty}2^{-k} ||(\chi_{E_{j}}d\sigma)^{\vee} (\chi_{E_{k}}d\sigma)^{\vee}  ||_{L^2 (F^3,dx)} \leq 2 \sum_{0\leq k \leq j}2^{-j-k}||(\chi_{E_{j}}d\sigma)^{\vee} (\chi_{E_{k}}d\sigma)^{\vee}  ||_{L^2 (F^3,dx)}.\]

From Lemma \ref{lem:bilinear}, we have:
\[||(\chi_{E_{j}}d\sigma)^{\vee} (\chi_{E_{k}}d\sigma)^{\vee}  ||_{L^2 (F^3,dx)} \leq 2^{1/2} |F|^{-5/2}(|E_j|^{1/2}|E_k|^2 + |E_j||E_k|)^{1/2}.\]
Using the fact that $(|E_j|^{1/2}|E_k|^2 + |E_j||E_k|)^{1/2} \leq (2 \max (|E_j|^{1/2}|E_k|^2, |E_j||E_k|))^{1/2} \leq 2^{1/2} (|E_j|^{1/4}|E_k| + |E_j|^{1/2}|E_k|^{1/2}),$ we obtain:
\[||(\chi_{E_{j}}d\sigma)^{\vee} (\chi_{E_{k}}d\sigma)^{\vee}  ||_{L^2 (F^3,dx)} \leq 2 |F|^{-5/2} \left(|E_j|^{1/4}|E_k| + |E_j|^{1/2}|E_k|^{1/2}\right).\]

Thus it suffices to show

\[ 2^{2} \sum_{0\leq k \leq j}2^{-j-k}\left(|E_j|^{1/4}|E_k| + |E_j|^{1/2}|E_k|^{1/2}\right) \leq (C')^2. \]

We consider the two sums $\sum_{0 \leq k \leq j} 2^{-j-k} |E_j|^{1/4}|E_k|$ and $\sum_{0\leq k \leq j} 2^{-j-k} |E_j|^{1/2}|E_k|^{1/2}$ separately. First we observe:
\[\sum_{0\leq k \leq j}2^{-j-k} |E_j|^{1/4}|E_k| \leq \sum_{0\leq k \leq j}2^{-j-k} 2^{j\cdot 2/5} |E_k| = \sum_{0\leq k \leq j} 2^{-j\cdot 3/5} 2^{-k}|E_k|.  \]

We can alternatively express this last quantity as:
\[\sum_{k=0}^{\infty} 2^{-k}|E_k| \left(\sum_{j=k}^{\infty} 2^{-j\cdot 3/5}\right) = \frac{1}{1-2^{-3/5}}\sum_{k=0}^{\infty}2^{-8/5 \cdot k} |E_k|.\]
Recalling that $\sum_{k=0}^{\infty}2^{-8/5 \cdot k} |E_k| =1$, we have:

\[  \sum_{0 \leq k \leq j} 2^{-j-k} |E_j|^{1/4}|E_k| \leq \frac{1}{1-2^{-3/5}}. \]

To bound the other sum, we simply use $|E_j| \leq 2^{j\cdot 8/5}$ and $|E_k| \leq 2^{k \cdot 8/5}$:
\[\sum_{0\leq k \leq j}2^{-j-k} |E_j|^{1/2}|E_k|^{1/2}\leq \sum_{k=0}^{\infty} 2^{-1/5 \cdot k} \left(\sum_{j=k}^{\infty} 2^{-1/5\cdot j}\right) = \frac{1}{\left(1-2^{-1/5}\right)\left(1-2^{-2/5}\right)}.\]

This shows that for $f$ of the form $f(\xi)= \sum_{j=0}^{\infty}2^{-j}\chi_{E_{j}}(\xi)$, we have $|| (f d\sigma)^{\vee} ||_{L^4 (F^3,dx)}^{2} \leq (C')^2 |F|^{-5/2}$,
for $C' = 2 \left( \frac{1}{\left(1-2^{-1/5}\right) \left(1-2^{-2/5}\right)}  + \frac{1}{1-2^{-3/5}} \right)^{1/2}$.
Therefore, for arbitrary complex-valued functions $f$, we have $||(f d\sigma)^\vee ||_{L^4(F^3, dx)} \leq C ||f||_{L^{8/5}(\mathcal{P}, d\sigma)}$,
where $C$ can be set to \\$4 \left( \frac{1}{\left(1-2^{-1/5}\right) \left(1-2^{-2/5}\right)}  + \frac{1}{1-2^{-3/5}} \right)^{1/2}$.
This concludes the proof of the theorem.
\end{proof}

\section{Restriction Theorem for the Paraboloid in Higher Dimensions}
We now consider the paraboloid $\mathcal{P}$ in $F^n$ for values of $n >3$.
We begin by presenting a combinatorial lemma.

\subsection{A Combinatorial Lemma}
\begin{lemma}\label{lem:setbounds} For any sets $A, B \subseteq \mathcal{P} \subseteq F^n$ where $n \geq 4$ is even or $n$ is odd and $|F| = q^m$ for some prime $q$ congruent to 3 modulo 4 with $m(n-1)$ not a multiple of 4, we have:
\[\sum_{\substack{a+b=c+d \\ a,c \in A \\ b,d \in B}} 1 \ll   |F|^{(n-2)/4} |A| |B|^{3/2} + |F|^{(n-2)/2}|A||B| + |F|^{-1}|A||B|^2,\]
which implies
\[ \left(\sum_{\substack{a+b=c+d \\ a,c \in A \\ b,d \in B }} 1 \right)^{1/2} \ll   |F|^{(n-2)/8} |A|^{1/2} |B|^{3/4} + |F|^{(n-2)/4}|A|^{1/2}|B|^{1/2} + |F|^{-1/2}|A|^{1/2}|B|.\]
\end{lemma}

\begin{proof} We follow the strategy used to prove Lemmas 7 and 8 in \cite{IK}, generalizing it appropriately to allow arbitrary $A$ and $B$ (in the \cite{IK} lemmas, $A = B$). We first note:
\[\sum_{\substack{a+b=c+d \\ a,c \in A \\ b,d \in B}} 1 \leq \sum_{\substack{a+b = c+d \\ a \in A \\ b,d \in B \\ c \in \mathcal{P}}} 1 = \sum_{\substack{a+b-d \in \mathcal{P} \\ a \in A \\ b,d \in B}} 1.\]
We can express a point $a \in \mathcal{A}$ as $a = (\underline{a}, \underline{a} \cdot \underline{a})$, for some $\underline{a} \in F_*^{n-1}$. Then, $a + b -d \in \mathcal{P}$ if and only if $\underline{a} \cdot \underline{b} -\underline{a} \cdot \underline{d}-\underline{b} \cdot \underline{d} + \underline{d} \cdot \underline{d} = 0$. We let $\delta$ denote the function on $F_*$ which is 1 when the input is 0 and is 0 otherwise. We then have:
\[\sum_{\substack{a+b-d \in \mathcal{P} \\ a \in A \\ b,d \in B}} 1 = \sum_{\substack{a \in A \\ b,d \in B}} \delta(\underline{a} \cdot \underline{b} - \underline{a} \cdot \underline{d} -\underline{b} \cdot \underline{d} + \underline{d} \cdot \underline{d}).\]

Now, for any value $t \in F_*$, we note that $\delta(t) = |F|^{-1} \sum_{s \in F_*} e(st)$. Thus, we have:
\[\sum_{\substack{a \in A \\ b,d \in B}} \delta(\underline{a} \cdot \underline{b} - \underline{a} \cdot \underline{d} -\underline{b} \cdot \underline{d} + \underline{d} \cdot \underline{d}) = |F|^{-1} \sum_{\substack{a \in A \\ b,d \in B}} \sum_{s \in F_*} e(s(\underline{a} \cdot \underline{b} - \underline{a} \cdot \underline{d} -\underline{b} \cdot \underline{d} + \underline{d} \cdot \underline{d})).\]
We note that when $s = 0$, the value of the sum over $a,b,d$ is equal to $|A||B|^2$. This gives us:
\[|F|^{-1} \sum_{\substack{a \in A \\ b,d \in B}} \sum_{s \in F_*} e(s(\underline{a} \cdot \underline{b} - \underline{a} \cdot \underline{d} -\underline{b} \cdot \underline{d} + \underline{d} \cdot \underline{d})) = |F|^{-1} |A||B|^2 + |F|^{-1}\sum_{\substack{a \in A \\ b,d \in B}}\sum_{\substack{s \in F_* \\ s \neq 0}} e(s(\underline{a} \cdot \underline{b} - \underline{a} \cdot \underline{d} -\underline{b} \cdot \underline{d} + \underline{d} \cdot \underline{d})).\]

We now consider upper bounding the quantity
\[\left| \sum_{\substack{a \in A \\ b,d \in B}}\sum_{\substack{s \in F_* \\ s \neq 0}} e(s(\underline{a} \cdot \underline{b} - \underline{a} \cdot \underline{d} -\underline{b} \cdot \underline{d} + \underline{d} \cdot \underline{d})) \right|^2 .\]
By the triangle inequality, this is:
\[\leq \left( \sum_{a \in A} \left| \sum_{\substack{b,d \in B \\ s \in F_* \\ s \neq 0}} e(s(\underline{a} \cdot \underline{b} - \underline{a} \cdot \underline{d} -\underline{b} \cdot \underline{d} + \underline{d} \cdot \underline{d}))\right|\right)^2.\]
By applying the Cauchy-Schwarz inequality to the sum over $a$, this is:
\[\leq |A| \sum_{a \in A} \left|\sum_{\substack{b,d \in B \\ s \in F_* \\ s \neq 0}} e(s(\underline{a} \cdot \underline{b} - \underline{a} \cdot \underline{d} -\underline{b} \cdot \underline{d} + \underline{d} \cdot \underline{d}))\right|^2.\]
Again employing the triangle inequality, we have:
\[\leq |A|\sum_{a \in A} \left( \sum_{ d \in B} \left| \sum_{\substack{b \in B \\ s \in F_* \\ s\neq 0}} e(s(\underline{a} \cdot \underline{b} - \underline{a} \cdot \underline{d} -\underline{b} \cdot \underline{d} + \underline{d} \cdot \underline{d}))\right|\right)^2.\]
By applying the Cauchy-Schwarz inequality to the sum over $d$, this is:
\[\leq |A| |B| \sum_{a \in A} \sum_{d \in B} \left| \sum_{\substack{b \in B \\ s \in F_* \\ s\neq 0}} e(s(\underline{a} \cdot \underline{b} - \underline{a} \cdot \underline{d} -\underline{b} \cdot \underline{d} + \underline{d} \cdot \underline{d}))\right|^2.\]
Since $B \subseteq \mathcal{P}$, this is:
\[\leq |A||B| \sum_{a \in A} \sum_{ \underline{d} \in F_*^{n-1}} \left|  \sum_{\substack{b \in B \\ s \in F_* \\ s\neq 0}} e(s(\underline{a} \cdot \underline{b} - \underline{a} \cdot \underline{d} -\underline{b} \cdot \underline{d} + \underline{d} \cdot \underline{d}))\right|^2.\]

For any $\underline{a} \in F_*^{n-1}$, we define the quantity $M(\underline{a})$ (with respect to $B$) as in \cite{IK}:
\[M(\underline{a}) := \sum_{ \underline{d} \in F_*^{n-1}} \left| \sum_{\substack{b \in B \\ s \in F_* \\ s\neq 0}} e(s(\underline{a} \cdot \underline{b} - \underline{a} \cdot \underline{d} -\underline{b} \cdot \underline{d} + \underline{d} \cdot \underline{d}))\right|^2.\]

In \cite{IK}, they prove for even $n \geq 4$ that\footnote{Actually, they state this for $\underline{a}$ such that $a \in B$ (since $A = B$ in their case), but their proof never uses that $a \in B$, so it extends without modification to all $\underline{a}$'s.}
\[M(\underline{a}) \ll |F|^{\frac{n+2}{2}} |B|^2 + |F|^{n} |B|\]
holds for all $\underline{a} \in F_*^{n-1}$.

This gives us:
\[|A||B| \sum_{a \in A} M(\underline{a}) \ll |A|^2|B| \left(|F|^{\frac{n+2}{2}} |B|^2 + |F|^n |B|\right).\]

By substituting this into our bounds above, we have that:
\[\sum_{\substack{a+b-d \in \mathcal{P} \\ a \in A \\ b,d \in B}} 1 \ll |F|^{-1}|A||B|^2 +|F|^{-1} \left(|A|^2|B| \left(|F|^{\frac{n+2}{2}} |B|^2 + |F|^n |B|\right)\right)^{1/2},\]
which is
\[ \ll |F|^{-1} |A||B|^2 + |F|^{-1} |F|^{\frac{n+2}{4}} |A| |B|^{3/2} + |F|^{-1} |F|^{\frac{n}{2}} |A||B| \]\[= |F|^{-1} |A||B|^2 + F^{\frac{n-2}{4}} |A||B|^{3/2} + |F|^{\frac{n-2}{2}} |A||B|.\]

For odd $n$ when $|F| = q^m$ for some prime $q$ congruent to 3 modulo 4 with $m(n-1)$ not a multiple of 4, they prove in \cite{IK} that
\footnote{Again, they state this for $\underline{a}$ such that $a \in B$ (since $A = B$ in their case), but their proof never uses that $a \in B$, so it extends without modification to all $\underline{a}$'s.}
\[M(\underline{a}) \ll |F|^{n} |B| + |F|^{\frac{n+1}{2}} |B|^2\]
holds for all $\underline{a} \in F_*^{n-1}$.

This gives us:
\[|A||B| \sum_{a \in A} M(\underline{a}) \ll |A|^2|B| \left(|F|^{n} |B| + |F|^{\frac{n+1}{2}} |B|^2\right).\]

By substituting this into our bounds above, we have that:
\[\sum_{\substack{a+b-d \in \mathcal{P} \\ a \in A \\ b,d \in B}} 1 \ll |F|^{-1}|A||B|^2 +|F|^{-1} |A||B|^{1/2}\left(|F|^{n} |B| + |F|^{\frac{n+1}{2}} |B|^2 \right)^{1/2},\]
which is
\[ \ll |F|^{-1} |A||B|^2 + |F|^{\frac{n-2}{2}} |A| |B| + |F|^{\frac{n-3}{4}} |A||B|^{3/2}.\]
We note that this is actually a somewhat better estimate than the lemma requires, since $|F|^{\frac{n-3}{4}} < |F|^{\frac{n-2}{4}}$.
\end{proof}

\subsection{Proof of Theorem 2}

We recall that $p := \frac{4n}{3n-2}$. We now prove Theorem 2, restated below in an equivalent formulation:

\begin{thm2} When $n \geq 4$ is even or when $n$ is odd and $|F| = q^m$ for a prime $q$ congruent to 3 modulo 4 such that $m(n-1)$ is not a multiple of 4, for every function $f: \mathcal{P} \rightarrow \mathbb{C}$ we have that:
\[||(f d\sigma)^\vee||_{L^4(F^n,dx)} \ll ||f||_{L^p(\mathcal{P},d\sigma)}.\]
\end{thm2}

\begin{proof}

Expanding the $L^4$ norm, we see that our task reduces to proving:

 \[||(f d\sigma)^\vee ||_{L^4(F^n, dx)}=|F|^{1-3n/4} \left(  \sum_{\substack{a+b = c+d \\ a,b,c,d \in \mathcal{P}}} f(a)f(b)\overline{f(c)}\overline{f(d)} \right)^{1/4}\]
\begin{equation}\label{eq:combinequality}
\leq C  |F|^{5/4 -1/2n-3n/4 }  \left(  \sum_{\xi \in \mathcal{P}} |f(\xi)|^{\frac{4n}{3n-2}} \right)^{\frac{3n-2}{4n}}=||f||_{L^{\frac{4n}{3n-2}}(\mathcal{P},d\sigma)}.
\end{equation}

We will find it more convenient to prove the equivalent formulation:

\[ |F|^{3n/2-2} || (fd\sigma)^{\vee} (fd\sigma)^{\vee} ||_{L^2(F^n,dx)} = \left(  \sum_{\substack{a+b = c+d \\ a,b,c,d \in \mathcal{P}}} f(a)f(b)\overline{f(c)}\overline{f(d)} \right)^{1/2}\]
 \[\ll |F|^{1/2-1/n} \left(  \sum_{\xi \in \mathcal{P}} |f(\xi)|^{\frac{4n}{3n-2}} \right)^{\frac{3n-2}{2n}}.\]

As before, we may assume that $f=\sum_{j=0}^{\infty} 2^{-j}\chi_{E_j}$ is a dyadic step function. Moreover, we will normalize $f$ to have $L^p$ norm $1$ in the counting measure. In other words, $\sum_{j=0}^{\infty} 2^{-pj}|E_j|=1$. It now suffices to prove

\[  |F|^{(3n/2-5/2+1/n)} \sum_{j=0}^{\infty} \sum_{k=0}^{\infty} 2^{-j}2^{-k}||(\chi_{E_j} d\sigma)^\vee (\chi_{E_k} d\sigma)^\vee  ||_{L^2(F^n,dx)} \ll 1.\]

Observe that \[|F|^{3n/2-2}\left|\left|(\chi_A d\sigma)^\vee (\chi_B d\sigma)^\vee \right|\right|_{L^2(F^n,dx)} = (\sum_{\substack{a+b = c+d \\ a,c \in A \\ b,d \in B}} 1)^{1/2}.\]
Hence, we can rewrite the inequality to be proved as:
\[ \sum_{j=0}^{\infty} \sum_{k=0}^{\infty} 2^{-j-k} (\sum_{\substack{a+b = c+d \\ a,c \in E_j \\ b,d \in E_k}} 1)^{1/2} \ll |F|^{1/2 - 1/n}.\]

By the symmetry of $j$ and $k$, it suffices to show
\[\sum_{j=0}^{\infty} \sum_{k=j}^{\infty} 2^{-j-k} (\sum_{\substack{a+b = c+d \\ a,c \in E_j \\ b,d \in E_k}} 1 )^{1/2} \ll |F|^{1/2 - 1/n}.\]

Lemma \ref{lem:setbounds} gives us an upper bound on the quantity $(\sum_{\substack{a+b = c+d \\ a,c \in E_j \\ b,d \in E_k}} 1 )^{1/2}$, but we can also obtain the simpler upper bound of $|E_j||E_k|^{1/2}$ by noting that for fixed values of $a, c \in E_j, b \in E_k$, there is at most one value of $d \in E_k$ which satisfies $a+b = c+d$.

We will split this sum into three pieces according to the following cases: 1. $|E_j| \leq |F|^{\frac{n-2}{2}}$, 2. $|E_k| \leq |F|^{\frac{n-2}{2}}$, and 3.  $|F|^{\frac{n-2}{2}} \leq |E_j|, |E_k| \leq |F|^{n-1}$.

We note that the union of these three cases covers all possibilities for subsets $E_j, E_k \subseteq \mathcal{P}$. We first consider case 1. We let $J$ denote the subset of $j$'s satisfying $|E_j| \leq |F|^{\frac{n-2}{2}}$. We also let $U$ denote the value $\log_2\left(|F|^{\frac{(n-2)(3n-2)}{8n}}\right)$. We note that when $j \leq U$, the bound $|E_j| \leq 2^{pj}$ is better than the bound $|E_j| \leq |F|^{\frac{n-2}{2}}$, and when $j>U$, the latter bound is better.

We consider:
\[\sum_{j \in J} \sum_{k=j}^{\infty} 2^{-j-k} \;(\sum_{\substack{a+b = c+d \\ a,c \in E_j \\ b,d \in E_k}} 1)^{1/2} \leq
\sum_{j \in J} \sum_{k=j}^{\infty} 2^{-j-k} |E_j||E_k|^{1/2}.\]
Here, we have used the simple upper bound of $|E_j||E_k|^{1/2}$ noted above. We can rewrite this as:
\[\sum_{\substack{j \in J \\ j\leq U}} \sum_{k=j}^\infty 2^{-j-k} |E_j||E_k|^{1/2} + \sum_{\substack{j \in J \\ j > U}} \sum_{k=j}^{\infty} 2^{-j-k} |E_j||E_k|^{1/2}.\]

Since we have assumed that $\sum_{j=0}^\infty 2^{-pj}|E_j| = 1$, we always have that $|E_j| \leq 2^{pj}$. Applying this to the first sum, we see that:
\[\sum_{\substack{j \in J \\ j\leq U}} \sum_{k=j}^\infty 2^{-j-k} |E_j||E_k|^{1/2} \ll \sum_{\substack{j \in J \\ j \leq U}}\sum_{k=j}^\infty 2^{-j-k} 2^{pj}2^{pk/2}.\]
Since $p/2 < 1$, the geometric sum over the $k$ values is convergent, and the value of the sum is bounded by a constant (depending on $n$) times its first term. Hence, we have that this is $\ll \sum_{\substack{j \in J \\ j \leq U}} 2^{j(3/2 p -2)}$. The exponent $(3/2 p -2)$ here is equal to $\frac{4}{3n-2} > 0$, so this geometric sum is bounded by a constant (depending on $n$) times its largest term, which is:
\[\ll 2^{U(3/2p - 2)} = |F|^{\frac{(n-2)(3n-2)(3/2 p -2)}{8n}} = |F|^{1/2 - 1/n}.\]

We now consider the sum:
\[\sum_{\substack{j \in J \\ j > U}} \sum_{k=j}^{\infty} 2^{-j-k}|E_j||E_k|^{1/2}.\]
Noting that $|E_j| \leq |F|^{\frac{n-2}{2}}$ and $|E_k| \leq 2^{pk}$, we have:
\[ \ll |F|^{\frac{n-2}{2}} \sum_{\substack{j \in J \\ j > U}} \sum_{k=j}^{\infty} 2^{-j-k} 2^{pk/2}.\]
Again, since $p/2 < 1$, the geometric sum of $k$ is convergent, so $\ll |F|^{\frac{n-2}{2}} \sum_{\substack{j \in J \\ j > U}} 2^{j(p/2 - 2)}.$
The exponent $p/2 -2$ is negative, so this geometric series in $j$ converges, and we have:
\[ \ll |F|^{\frac{n-2}{2}} 2^{U(p/2-2)} = |F|^{\frac{n-2}{2} + \frac{(n-2)(3n-2)(p/2-2)}{8n}} = |F|^{1/2 - 1/n}.\] This concludes our proof for values of $j \in J$.

We note cases 1 and 2 are symmetric, and we are left with considering case 3, where $j,k \notin J$. In this case, we apply the bound provided by Lemma \ref{lem:setbounds} with $A := E_k$ and $B:= E_j$:
\[\sum_{j\notin J} \sum_{\substack{k \notin J \\ k\geq j}} 2^{-j-k} \left(\sum_{\substack{a+b = c+d \\ a,c \in E_j \\ b,d \in E_k}} 1\right)^{1/2}  \ll \sum_{j \notin J} \sum_{\substack{k \notin J \\ k \geq j}} 2^{-j-k}|F|^{(n-2)/8} |E_k|^{1/2} |E_j|^{3/4} \]
\[+ \sum_{j \notin J} \sum_{\substack{k \notin J \\ k \geq j}} 2^{-j-k}|F|^{(n-2)/4}|E_k|^{1/2}|E_j|^{1/2} + \sum_{j \notin J} \sum_{\substack{k \notin J \\ k \geq j}} 2^{-j-k} |F|^{-1/2}|E_k|^{1/2}|E_j|.\]

We observe that
\[|F|^{(n-2)/4} |E_k|^{1/2}|E_j|^{1/2} \leq |F|^{(n-2)/8} |E_k|^{1/2} |E_j|^{3/4}\]
as long as $|F|^{(n-2)/2} \leq |E_j|$, i.e. whenever $j \notin J$. Thus, it suffices to bound the first and third sums. We consider the first sum with the change of variable $j =k- \ell$:
\[\sum_{j \notin J} \sum_{\substack{k \notin J \\ k \geq j}} 2^{-j-k}|F|^{(n-2)/8} |E_k|^{1/2} |E_j|^{3/4} \leq |F|^{(n-2)/8} \sum_{\ell = 0}^{\infty} 2^{\ell} \sum_{\substack{k \geq \ell \\ k \notin J}} 2^{-2k} |E_k|^{1/2} |E_{k-\ell}|^{3/4}.\]
We define the values $0\leq c_k \leq 1$ by $|E_k| = c_k 2^{kp}$. We recall that $\sum_{k=0}^{\infty} 2^{-kp}|E_k| = 1$ implies that the sum $\sum_{k=0}^{\infty} c_k$ converges (in particular, it is equal to 1). We can rewrite the quantity above as:
\[ = |F|^{(n-2)/8} \sum_{\ell = 0}^{\infty} 2^{\ell (1-3/4p)} \sum_{\substack{k \geq \ell \\ k \notin J}} c_k^{1/2}\; c_{k-\ell}^{3/4}\; 2^{k(5/4p-2)}.\]

For $k \notin J$, we have that $|F|^{\frac{n-2}{2}} \leq |E_k| \leq 2^{pk}$, so $2^k \geq |F|^{\frac{n-2}{2p}}$. We note that $5/4p-2 = \frac{4-n}{3n-2} \leq 0$, so
\[2^{k(5/4p-2)} \leq |F|^{\frac{(n-2)(5/4p-2)}{2p}} = |F|^{\frac{(n-2)(4-n)}{8n}}.\] We thus have:
\[|F|^{(n-2)/8} \sum_{\ell = 0}^{\infty} 2^{\ell (1-3/4p)} \sum_{\substack{k \geq \ell \\ k \notin J}} c_k^{1/2}\; c_{k-\ell}^{3/4}\; 2^{k(5/4p-2)} \ll |F|^{\frac{n-2}{8} + \frac{(n-2)(4-n)}{8n}}\sum_{\ell = 0}^{\infty} 2^{\ell(1-3/4p)}\sum_{\substack{k \geq \ell \\ k \notin J}} c_k^{1/2}c_{k-\ell}^{3/4}.\]

We note that $\frac{n-2}{8} + \frac{(n-2)(4-n)}{8n} = 1/2 - 1/n$. For each fixed $\ell$, we apply the Cauchy-Schwarz inequality to the inner sum over $k$ to obtain:
\[ \ll |F|^{1/2-1/n} \sum_{\ell=0}^\infty 2^{\ell(1-3/4p)} \left( \sum_{k=\ell}^{\infty} c_k\right)^{1/2} \left(\sum_{k = \ell}^\infty c_{k-\ell}^{3/2}\right)^{1/2}.\]
Since $0 \leq c_k \leq 1$, we have that $c_{k-\ell} ^{3/2} \leq c_{k-\ell}$ for all $k,\ell$, so $\sum_{k=0}^{\infty} c_k = 1$ implies this is:
\[ \ll |F|^{1/2-1/n} \sum_{\ell = 0}^{\infty} 2^{\ell(1-3/4p)}.\]
Now, $1-3/4p = \frac{-2}{3n-2} < 0$, so this geometric sum over $\ell$ converges, and we obtain $\ll |F|^{1/2-1/n}$, as desired.

We now consider the third sum, $\sum_{j \notin J} \sum_{\substack{k \notin J \\ k \geq j}} 2^{-j-k} |F|^{-1/2}|E_k|^{1/2}|E_j|$. We define the value $U$ to be: $U:= \log_2\left(|F|^{(n-1)/p}\right)$. We note that when $j \leq U$, the bound $|E_j|\leq 2^{pj}$ is better than the bound $|E_j| \leq |F|^{n-1}$, and when $j > U$, the latter bound is better.

Our sum is then:
\[ \ll |F|^{-1/2} \sum_{j\leq U} \sum_{k \geq j} 2^{-k-j} |E_j||E_k|^{1/2} + |F|^{-1/2} \sum_{j \geq U} \sum_{k \geq j}2^{-k-j} |E_j||E_k|^{1/2}.\]
For the first of these two sums, we use $|E_j| \leq 2^{pj}$ and $|E_k| \leq 2^{pk}$:
\[|F|^{-1/2} \sum_{j\leq U} \sum_{k \geq j} 2^{-k-j} |E_j||E_k|^{1/2} \ll |F|^{-1/2} \sum_{j \leq U} \sum_{k \geq j} 2^{j(p-1)}2^{k(p/2-1)}.\] Since $p/2 < 1$, the geometric sum over $k$ is convergent, and we get $\ll |F|^{-1/2} \sum_{j\leq U} 2^{j(3/2p-2)}$.
Now, $3/2p -2  = \frac{4}{3n-2}>0$, so this is:
\[\ll |F|^{-1/2} 2^{U(3/2p-2)} = |F|^{-1/2} |F|^{\frac{(n-1)(3/2p-2)}{p}} = |F|^{1/2-1/n}.\]

To bound the sum for values of $j > U$, we use that $|E_j| \leq |F|^{n-1} = |\mathcal{P}|$ and $|E_k| \leq 2^{kp}$:
\[|F|^{-1/2} \sum_{j \geq U} \sum_{k \geq j} 2^{-k-j}|E_j||E_k|^{1/2} \ll |F|^{-1/2 + n-1} \sum_{j \geq U} \sum_{k \geq j} 2^{-j}2^{k(p/2-1)}.\]
The geometric sum over $k$ is convergent, so we have $\ll |F|^{-1/2 + n-1} \sum_{j \geq U} 2^{j(p/2-2)}.$ The geometric sum over $j$ is now also convergent, so we have:
\[\ll |F|^{-1/2 + n-1}|F|^{\frac{(n-1)(p/2-2)}{p}} = |F|^{1/2-1/n}.\]
This concludes the proof of the theorem.
\end{proof}

\section{Acknowledgments} We would like to thank Jeff Vaaler for helpful discussions.

\texttt{A. Lewko, Department of Computer Science, The University of Texas at Austin}

\textit{alewko@cs.utexas.edu}
\vspace*{0.5cm}

\texttt{M. Lewko, Department of Mathematics, The University of Texas at Austin}

\textit{mlewko@math.utexas.edu}


\begin{thebibliography}{5}

\bibitem{BCGW}
A. Carbery, Harmonic Analysis on Vector Spaces over Finite Fields, \url{http://www.maths.ed.ac.uk/uploads/assets/7_fflpublic.pdf}

\bibitem{IK}
A. Iosevich and D. Koh, Extension theorems for paraboloids in the finite field setting. Mathematische Zeitschrift, 266, (2010) no.2 471-487.

\bibitem{IKs}
A. Iosevich and D. Koh, Extension theorems for spheres in the finite field setting. Forum Math. 22 (2010), no. 3, 457-483.

\bibitem{IKq}
A. Iosevich and D. Koh, Extension theorems for the Fourier transform associated with non-degenerate quadratic surfaces in vector spaces over finite fields. Illinois Mathematics Journal, 52, no. 2 (2009), 611-628.

\bibitem{MT}
G. Mockenhaupt and T. Tao, Restriction and Kakeya phenomena for finite fields. Duke Math. J. 121 (2004), no. 1, 35-74.

\bibitem{T}
T. Tao, Recent Progress on the Restriction Conjecture. Fourier analysis and convexity, 217-243, Appl. Numer. Harmon. Anal., Birkhuser Boston, Boston, MA, 2004.
\end{thebibliography}
\end{document}